\documentclass[submission,copyright,creativecommons]{eptcs}
\providecommand{\event}{GASCom 2026}

\usepackage{iftex}

\ifpdf
  \usepackage{underscore}
  \usepackage[T1]{fontenc}
\else
  \usepackage{breakurl}
\fi

\usepackage{amsthm}
\usepackage{amsmath}
\usepackage{amssymb}
\usepackage{hyperref}
\usepackage{graphicx}
\usepackage{tikz}
\usepackage[english]{babel}
\usepackage{subcaption}
\usepackage{adjustbox}
\usepackage{float}
\usepackage{bbm}
\usepackage{caption}
\usepackage{makecell}

\usepackage[normalem]{ulem}
\usepackage{todonotes}

\newtheorem{proposition}{Proposition}
\newtheorem{theorem}{Theorem}
\newtheorem{lemma}{Lemma}

\newtheorem{proof*}{Proof}
\newtheorem{definition}{Definition}

\title{Penrose P2 Tilings: A Study of Fully Leafed Induced Subtrees}

\author{Mathieu Cloutier\thanks{Supported by NSERC and FRQNT.}
  \institute{Université du Québec à Trois-Rivières\\ Trois-Rivières, Québec}
  \email{mathieu.cloutier4@uqtr.ca}
\and
Alain Goupil\thanks{Supported by NSERC.}
  \institute{LACIM, Montréal, Québec}
  \institute{Université du Québec à Trois-Rivières\\ Trois-Rivières, Québec}
  \email{alain.goupil@uqtr.ca}
\and 
Alexandre Blondin Massé\thanks{Supported by NSERC.}
  \institute{LACIM, Montréal, Québec}
  \institute{Université du Québec à Montréal\\ Montréal, Québec}
  \email{blondin\_masse.alexandre@uqam.ca}
}

\def\titlerunning{Penrose P2 Tilings: A Study of Fully Leafed Induced Subtrees}
\def\authorrunning{M. Cloutier, A. Goupil \& A. Blondin Massé}

\begin{document}
\maketitle

\begin{abstract}
We present new results about fully leafed induced subtrees in Penrose P2 tilings, also known as kites and darts tilings. We first determine the graph structure of these subtrees and show that they are caterpillars up to an appendix of at most six tiles. In other words, if we remove their degree one vertices, then they are path graphs up to an additional connected path of at most two tiles. We then study bi-infinite fully leafed induced caterpillars in P2 tilings and their geometric properties. In particular, we refute the conjecture proposed by C. Porrier, A. Goupil and A. Blondin Massé that there is a unique bi-infinite fully leafed caterpillar in Penrose P2 tilings.
\end{abstract}

\section{Introduction}
Let us play a mathematical game. Consider a tiling $\mathcal{T}$ and a positive integer $n$. For a patch of tiles $P$ in $\mathcal{T}$, a tile of $P$ is said to be a degree one tile for $P$ if it is connected to only one other tile of $P$. These degree one tiles can be viewed as extremities of $P$. The goal of the game is to find a connected and acyclic patch of $n$ tiles in $\mathcal{T}$ that maximizes the number of degree one tiles. These winning patches correspond to what we call fully leafed induced subtrees. These structures have been studied in periodic tilings \cite{blondin2018fully, masse2018saturated}. In this paper, we are interested in describing these structures in a family of aperiodic tilings. More precisely, the tilings considered in this paper are the Penrose kites and darts tilings, which we denote by the P2 tilings. Interestingly, fully leafed induced subtrees can be used in chemistry to study adsorption on surfaces with a maximal number of boundary sites \cite{Madras2017}. Adsorption is the process by which particles adhere to a surface. Since Penrose tilings model quasicrystals, the results presented here may be of interest in the study of adsorption on quasicrystal surfaces \cite{madison2013symmetry, mcgrath2010surface, senechal1996quasicrystals}.

Section 2 presents the definitions needed for this work. In Section 3, we establish the graph structure of fully leafed induced subtrees in Penrose P2 graphs. Section 4 describes the construction of fully leafed induced subtrees, taking into account the geometry of the tilings in which they are embedded. Section 5 studies the construction of bi-infinite fully leafed induced caterpillars. 

Complete proofs of the new results presented here, as well as additional results on the same topic, can be found in the extended version of this paper \cite{cloutier2026fullyleafedinducedsubtrees}.

\section{Definitions}

Basic definitions from graph theory and tilings can be found in \cite{diestel2025graph} and \cite{grunbaum1987tilings}. The vertices of degree one in a tree $T$ are called the \textbf{leaves} of $T$.
For a graph $G$, the graph $G'$ obtained by removing the leaves from $G$ is called the \textbf{derived} graph of $G$. A \textbf{caterpillar} in $G$ is a subgraph $\mathcal{C}$ of $G$ such that $\mathcal{C}'$ is a chain.

\begin{definition}{(\cite{masse2018saturated})}
    A \textbf{fully leafed induced subtree} of order $n$ of a graph $G$ is an induced subtree of $G$ that maximizes the number of leaves among all induced subtrees of order $n$ in $G$.
\end{definition}

A Penrose P2 tiling, shortly a P2 tiling, is a tiling of the plane constructed by two sorts of tiles, called \textbf{prototiles}, described in Figure \ref{fig:prototuiles}. We know that P2 tilings are aperiodic \cite{grunbaum1987tilings}. A \textbf{patch} is a finite connected set of tiles that respects the matching rules. We say a patch $P$ is \textbf{valid} if, when $P$ is extended by the tiles that are forced to be placed according to the matching rules, the resulting set of tiles is still a patch.

We call \textbf{P2-graph} the dual graph of a P2 tiling. So we often use the term \textit{tile} to denote a vertex of a P2-graph. All graphs considered in this paper are induced subgraphs of P2-graphs, so we may not mention it every time. In this paper, we want to describe the graph structure of the fully leafed induced subtrees in P2-graphs. We know that P2 tilings have the local isomorphism property, which implies that every valid P2 patch appears in any P2 tiling \cite{grunbaum1987tilings}. Then, by the local isomorphism property, the fully leafed induced subtrees do not depend on a specific P2 tiling in which they would be embedded. 

\begin{figure}[H]
    \centering
    \begin{subfigure}{0.15\textwidth}
        \centering
        \includegraphics[width=\textwidth]{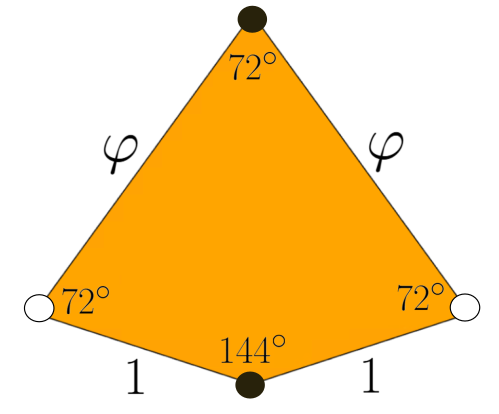}
        \caption{Kite}
    \end{subfigure} 
    \hspace{2cm}
    \begin{subfigure}{0.15\textwidth}
        \centering
        \includegraphics[width=\textwidth]{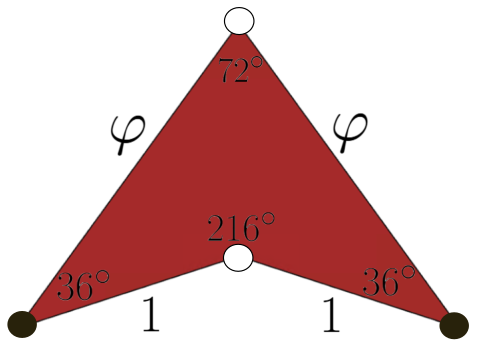}
        \caption{Dart}
    \end{subfigure}
    \caption{Prototiles of P2 tilings. Tiles are edge-adjacent only when black (resp. white) vertices coincide. The golden ratio $\varphi = \frac{1+\sqrt{5}}{2}$ is the ratio of long to short edges in both types of tile.}
    \label{fig:prototuiles}
\end{figure}

Two important patches of P2 tilings will be used throughout this paper: the star and the sun. They are presented in Figure \ref{fig:Two P2 patches}. 

\begin{figure}[H]
    \centering
    \begin{subfigure}{0.1\textwidth}
        \centering
        \includegraphics[width=\textwidth]{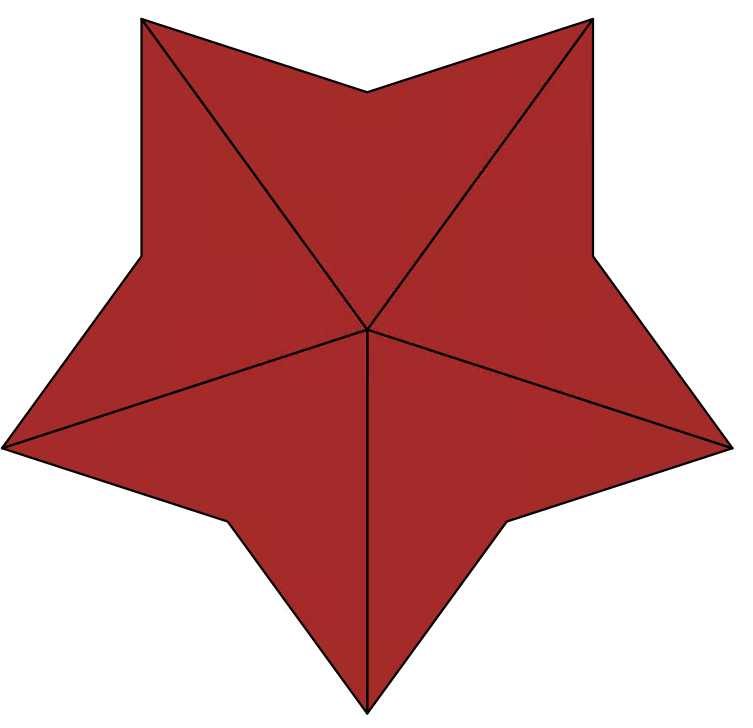}
        \caption{Star}
    \end{subfigure}
    \hspace{2cm}
    \begin{subfigure}{0.1\textwidth}
        \centering
        \includegraphics[width=\textwidth]{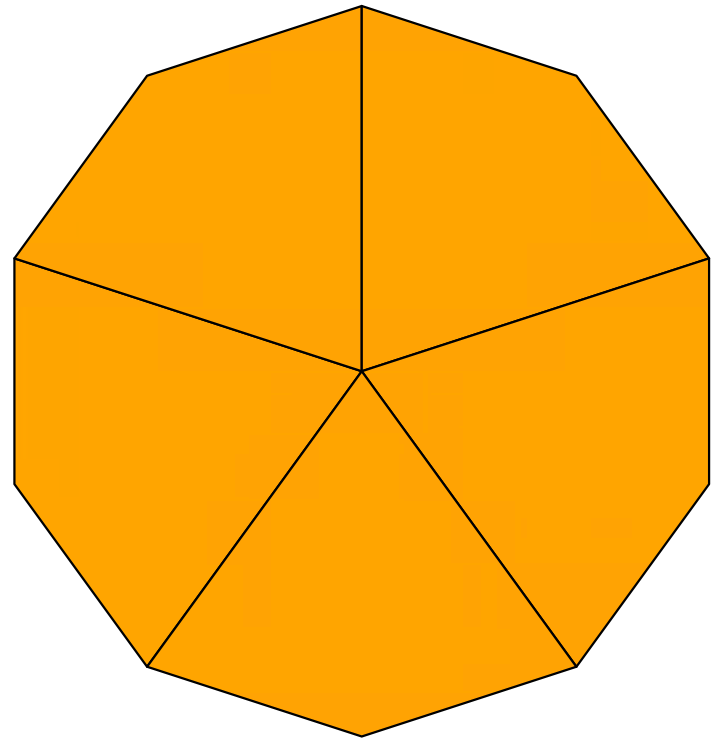}
        \caption{Sun}
    \end{subfigure}
    \caption{Two P2 patches}
    \label{fig:Two P2 patches}
\end{figure}

An important concept of Penrose tilings is inflation \cite{grunbaum1987tilings}. Consider a kite (resp. a dart), that we have outlined by a blue dashed contour in Figure \ref{fig:inflation cerf-volant} (resp. \ref{fig:inflation fléchette}). The \textbf{inflation} of this kite (resp. this dart) is the patch of 4 tiles (resp. 3 tiles) that are $\varphi$ times smaller than the original tiles, represented in brown and orange in Figure \ref{fig:inflation cerf-volant} (resp. \ref{fig:inflation fléchette}). The inflation of a P2 patch $P$ is defined as the patch obtained by inflating all the P2 tiles of $P$. By doing an infinite number of inflations from a given patch and by rescaling the dimension at each inflation, we can construct a P2 tiling.

\begin{figure}[H]
    \centering
    \begin{subfigure}{0.25\textwidth}
        \centering
        \includegraphics[width=0.6\textwidth]{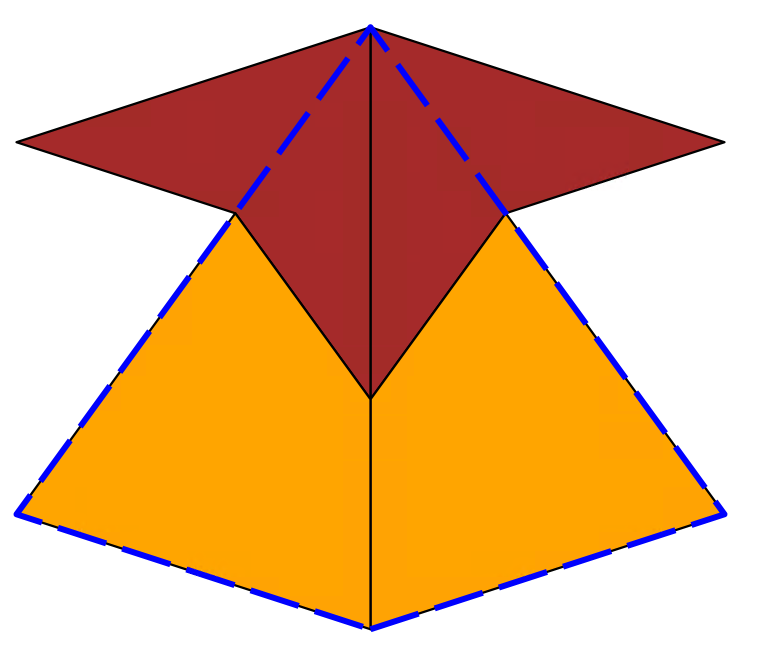}
        \caption{Inflation of a kite}
        \label{fig:inflation cerf-volant}
    \end{subfigure} 
    \hspace{2cm}
    \begin{subfigure}{0.25\textwidth}
        \centering
        \includegraphics[width=0.6\textwidth]{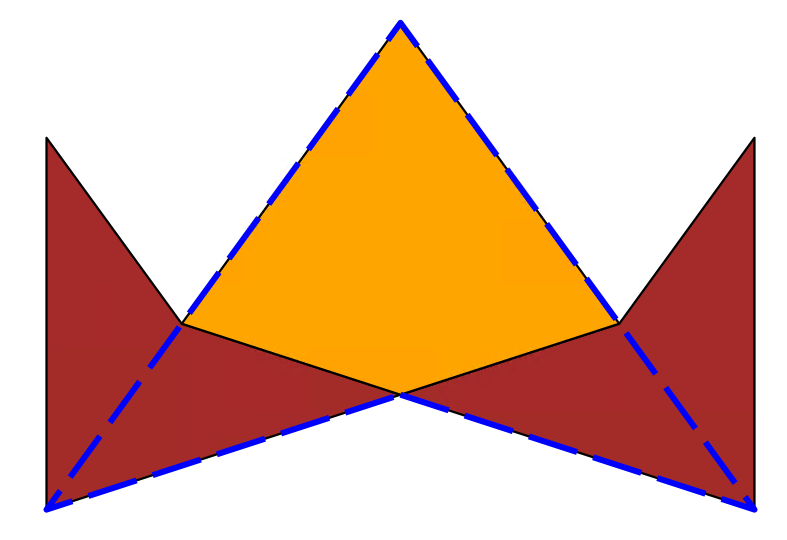}
        \caption{Inflation of a dart}
        \label{fig:inflation fléchette}
    \end{subfigure}
    \caption{Inflation of P2 tiles}
\end{figure}

\section{Graph structure}

It has been shown that every fully leafed induced subtree of a P2-graph that has 8 internal tiles (i.e., tiles of degree 2 or more) or fewer is a subcaterpillar of a caterpillar shown in Figure \ref{fig:chenilles premières} \cite{porrier2023leaf}. In this paper, we call a \textbf{prime caterpillar} such a fully leafed induced subtree of 8 internal tiles. We notice that internal tiles of prime caterpillars are all of degree 3. It has been verified that a fully leafed induced subtree of 9 internal tiles or more must have at least one tile of degree 2 \cite{porrier2023leaf}. We know that an induced subtree of a P2-graph is fully leafed when its number of degree three tiles is maximal, or equivalently, when its number of degree two tiles is minimal, because a P2 tile of a tree has degree 3 or less \cite{cloutier2026fullyleafedinducedsubtrees, porrier2019leaf, porrier2023leaf}. Since prime caterpillars are the largest fully leafed induced subtrees with all internal tiles of degree 3, in order to construct a fully leafed induced subtree, we want to maximize the number of prime caterpillars in it. Thus, prime caterpillars are primitive structures that we want to join together to form fully leafed induced subtrees.

\begin{definition}[\cite{masse2018saturated,porrier2023leaf}]
Let $I_1$, $I_2$ be two fully leafed induced subtrees and $t$  a tile such that $t$ is a leaf of $I_1$ and $I_2$, $I_1\cup I_2$ is a fully leafed induced subtree and $I_1\cap I_2=\{t\}$. Then we say that $I=I_1\cup I_2$ is a \textbf{grafting} of $I_1$ and $I_2$ at $t$ and we write $I=I_1\diamond _t I_2$, or just $I=I_1\diamond I_2$.
\end{definition}

\begin{figure}[H]
    \centering
    \begin{subfigure}{0.16\textwidth}
        \centering
        \includegraphics[width=\textwidth]{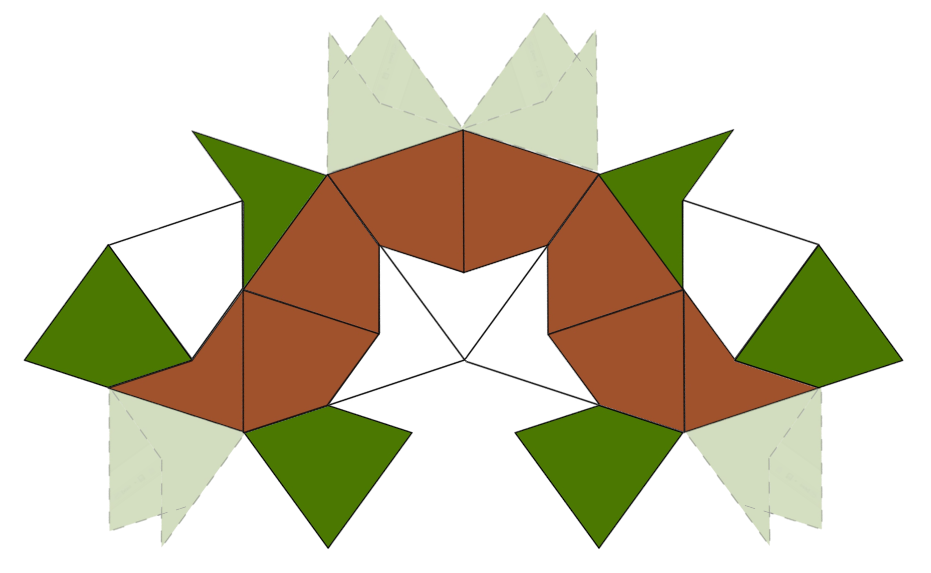}
        \caption{$PC_1$}
        \label{fig:CP1}
    \end{subfigure} \hfill
    \begin{subfigure}{0.15\textwidth}
    \includegraphics[width=\textwidth, trim=0cm 0.4cm 0cm 0cm, clip]{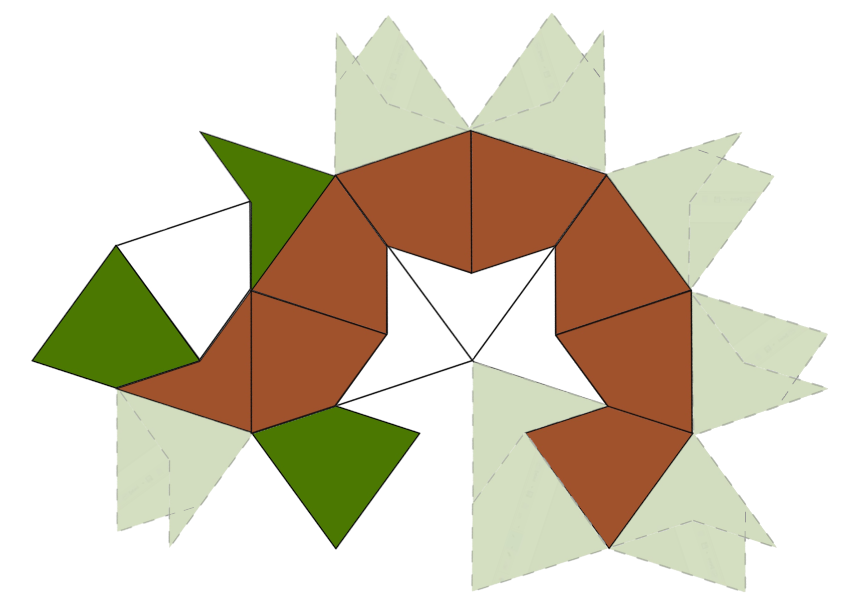}
    \caption{$PC_2$}
    \label{fig:CP2}
    \end{subfigure} \hfill
    \begin{subfigure}{0.16\textwidth}
    \includegraphics[width=\textwidth, trim=0cm 0cm 0cm 0cm, clip]{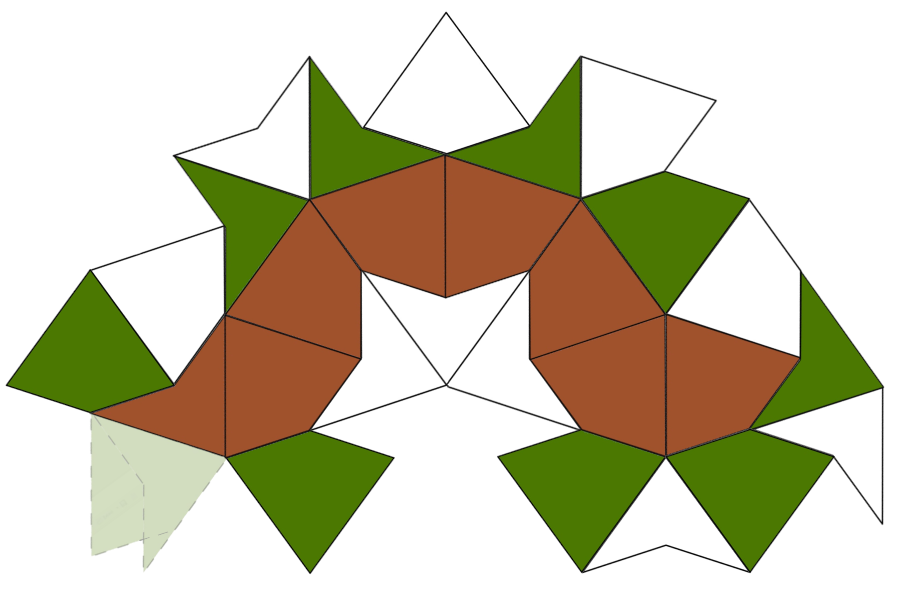}
    \caption{$PC_3$}
    \label{fig:CP3}
    \end{subfigure} \hfill
    \begin{subfigure}{0.15\textwidth}
    \includegraphics[width=\textwidth, trim=0cm 0.5cm 0cm 0cm, clip]{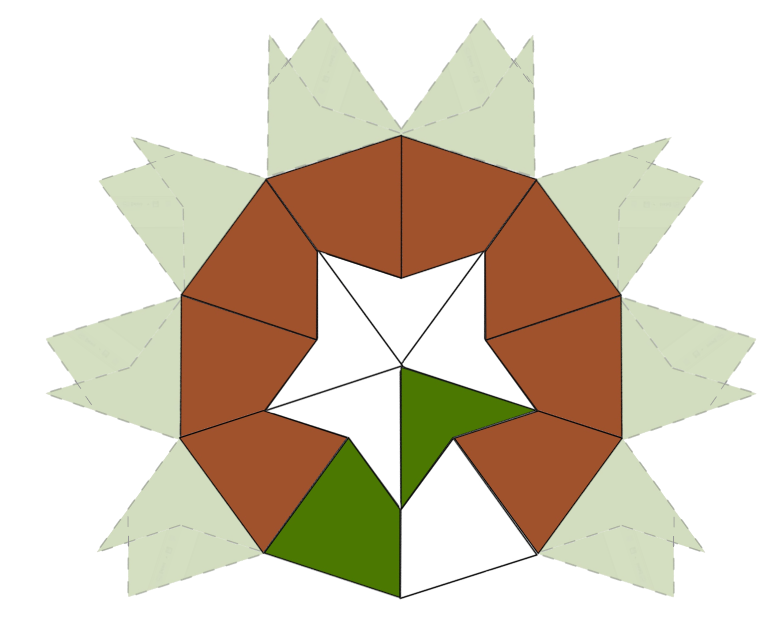}
    \caption{$PC_4$}
    \label{fig:CP4}
    \end{subfigure} \hfill
    \begin{subfigure}{0.16\textwidth}
    \includegraphics[width=\textwidth]{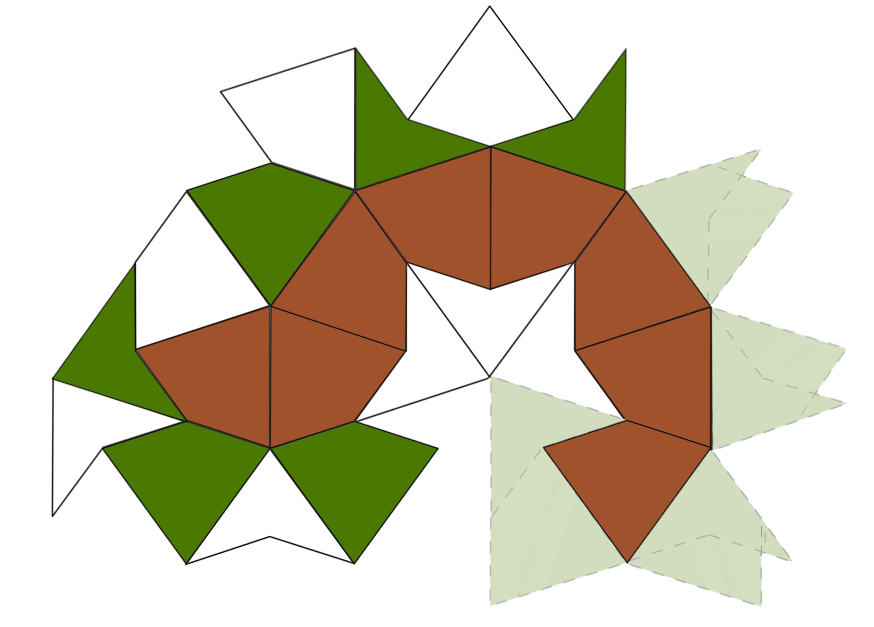}
    \caption{$PC_5$}
    \label{fig:CP5}
    \end{subfigure} \hfill
    \begin{subfigure}{0.16\textwidth}
    \includegraphics[width=\textwidth, trim=0.5cm 0cm 0cm 0cm, clip]{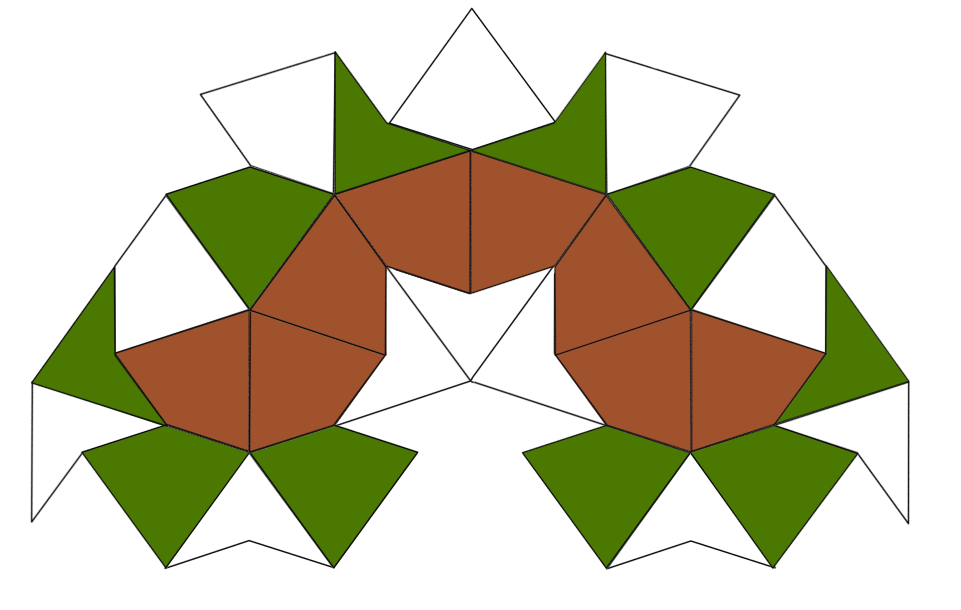}
    \caption{$PC_6$}
    \label{fig:CP6}
    \end{subfigure} \hfill
    \caption{The six possible prime caterpillars up to isometry and choice of leaves. Light green tiles are choices for leaves.}
    \label{fig:chenilles premières}
    
\end{figure}

There are two possibilities for extending a prime caterpillar $PC_k$ into a bigger fully leafed induced subtree. We can  extend it into a bigger caterpillar or into a tree that is not a caterpillar. In the second case, the fully leafed induced subtree grafted to $PC_k$ is called an \textbf{appendix} of $PC_k$. We have proved the following results.

\begin{lemma}[\cite{cloutier2026fullyleafedinducedsubtrees}]
    Let $\mathcal{C}=A\diamond PC_k$ be a fully leafed induced subtree obtained by grafting the appendix $A$ to the prime caterpillar $PC_k$. Then $A$ has at most 2 internal tiles.
\end{lemma}
\begin{theorem}[\cite{cloutier2026fullyleafedinducedsubtrees}] \label{thm2}
    Every fully leafed induced subtree of a P2-graph belongs to one of the following structures:
    \begin{enumerate}
        \item[1.] A subcaterpillar of a prime caterpillar;
        \item[2.] A caterpillar obtained by grafting several prime caterpillars with at most one proper subcaterpillar of a prime caterpillar structure;
        \item[3.] A caterpillar obtained by grafting several prime caterpillars, to which an appendix of at most 2 internal tiles and 4 leaves is grafted.
    \end{enumerate}
\end{theorem}

Thus, in order to study fully leafed induced subtrees in P2 tilings, it is sufficient to study caterpillars obtained by grafting prime caterpillars, since any fully leafed induced subtree can be obtained from such a caterpillar by adding a small number of tiles. Caterpillars obtained by grafting prime caterpillars are precisely the fully leafed induced subtrees that have the property of saturation \cite{cloutier2026fullyleafedinducedsubtrees}. Let us explain what it means. The leaf function $L_{P2}(n)$ gives the number of leaves of a fully leafed induced subtree of order $n$. We give here an exact expression for $L_{P2}(n)$ that corrects the one given in \cite{porrier2023leaf}.
\begin{proposition}[\cite{cloutier2026fullyleafedinducedsubtrees}] \label{prop:formule close} 
$L_{P2}(n)=
\begin{cases}
0 & \text{if } 0 \leq n \leq 1, \\ 
\left\lfloor \dfrac{n}{2} \right\rfloor + 1 & \text{if } 2 \leq n \leq 18, \\ 
8 \left\lfloor \dfrac{n}{17} \right\rfloor + \left\lfloor \dfrac{n \bmod 17}{2} \right\rfloor + 1 + \mathbbm{1}(n \bmod 17 = 1) & \text{if } n \geq 19.
\end{cases}$
\end{proposition}

Let $\overline{L_{P2}}$ be the lowest linear function that upper bounds $L_{P2}$. We say that a fully leafed induced subtree of order $n$ is \textbf{saturated}
if $L_{P2}(n)=\overline{L_{P2}}(n)$ \cite{blondin2018fully}.

\section{Prime caterpillar graftings}

We study more precisely here how prime caterpillars can be grafted together. The only two possibilities for the grafting of two prime caterpillars are illustrated in Figure \ref{fig:greffages premiers HBS} \cite{cloutier2026fullyleafedinducedsubtrees}.

\begin{figure}[H]
    \centering
    \begin{subfigure}{0.23\textwidth}
        \centering
        \includegraphics[width=\textwidth]{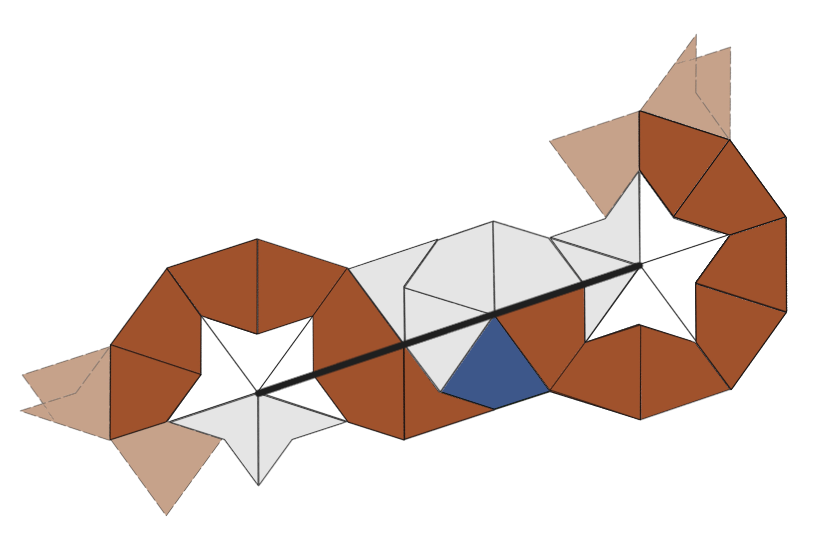}
        \caption{}
    \end{subfigure} 
    \begin{subfigure}{0.18\textwidth}
        \centering
        \includegraphics[width=\textwidth]{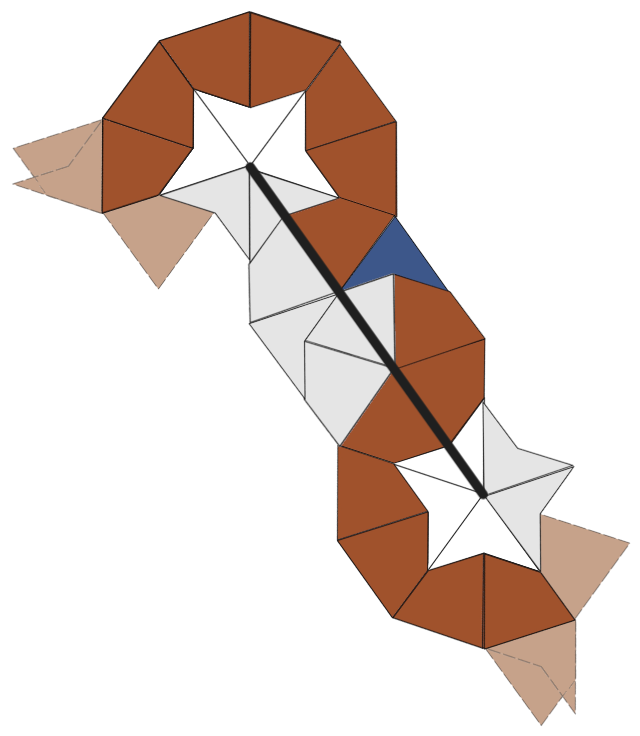}
        \caption{}
    \end{subfigure}
    \caption{The only two derived paths of graftings of two prime caterpillars. Light brown tiles are possibilities for degree three tiles. Blue tiles have degree $2$. We add a segment that joins the centers of the stars adjacent to the prime caterpillars.}
    \label{fig:greffages premiers HBS}
\end{figure}

For a given P2 tiling $T$, if we connect the centers of the stars in $T$ with edges as in Figure \ref{fig:greffages premiers HBS}, we obtain a new graph that we call a \textbf{star-graph}. The tilings with vertices and edges of the star-graphs are precisely the HBS tilings \cite{porrier2023hbs}. We illustrate a region of an HBS tiling in Figure \ref{fig:HBS}. A vertex of a star-graph is colored red, green or blue if it is the vertex in the center of a P2 star adjacent to 0, 1 or 2 suns, respectively. Then, the problem of finding fully leafed caterpillars in P2-graphs becomes the problem of finding paths in star-graphs that determine fully leafed caterpillars.

\begin{figure}[H]
    \centering
    \includegraphics[width=0.6\textwidth]{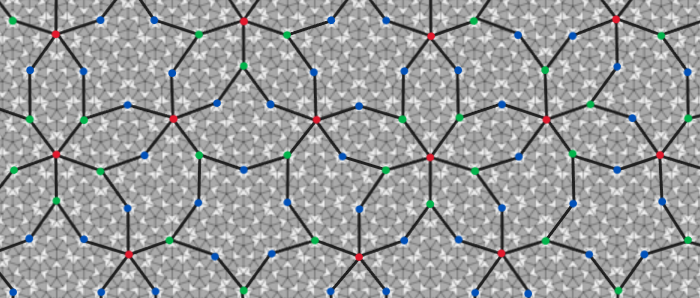}
    \caption{A region of an HBS tiling superimposed on a P2 tiling. Vertices at the center of a P2 star are color-coded red, green and blue depending on whether the star is adjacent to 0, 1, or 2 suns, respectively.}
    \label{fig:HBS}
\end{figure} 

Let $\mathcal{C}$ be a fully leafed caterpillar with star-graph path $p_\mathcal{C}$ that contains at least two prime caterpillars $C_1$ and $C_2$. Let $p'_\mathcal{C}$ be the extension of $p_\mathcal{C}$ obtained by extending the two end segments of $p_\mathcal{C}$ into two Euclidean half lines, so $p'_\mathcal{C}$ splits the Euclidean plane into two disjoint regions. We say that $C_1$ and $C_2$ \textbf{lie on the same side} (resp. \textbf{lie on opposite sides}) of $p_\mathcal{C}$ if the two internal tiles of $C_1'''$ and $C_2'''$ all lie (resp. do not lie) on a same side of $p'_\mathcal{C}$. The two internal tiles of $C_1'''$ and $C_2'''$ are the tiles that are at the center of $C_1$ and $C_2$ respectively. An example that illustrates this definition can be found in \cite{cloutier2026fullyleafedinducedsubtrees}.

\begin{proposition}[\cite{cloutier2026fullyleafedinducedsubtrees}] \label{prop: alternance}
If $\mathcal{C}$ is a fully leafed caterpillar, for $C_j$ and $C_{j+1}$ two grafted prime caterpillars in $\mathcal{C}$, $C_j$ and $C_{j+1}$ are on opposite sides of the star-graph path of $\mathcal{C}$.
\end{proposition}

From Figure \ref{fig:greffages premiers HBS}, we observe that every prime caterpillar $C$ determines a path $p_C$ of two connected edges in a star-graph. We define the \textbf{angle} $\alpha$ of a prime caterpillar $C$ as the Euclidean angle determined by $p_C$ and measured on the side of $p_C$ where the two internal tiles of $C'''$ lie. According to this definition, the angle measured in radians of the prime caterpillars $PC_1$, $PC_3$ and $PC_6$ is $\frac{4\pi}{5}$, the angle of $PC_2$ and $PC_5$ is $\frac{6\pi}{5}$ and the angle of $PC_4$ is $\frac{8\pi}{5}$ \cite{cloutier2026fullyleafedinducedsubtrees}.

We claim that if the star-graph path of a fully leafed saturated caterpillar $\mathcal{C}$ contains at least one pair of edges forming an angle of $\frac{8\pi}{5}$, then this star-graph path uniquely determines the saturated caterpillar $\mathcal{C}$. Indeed, an angle of measure $\frac{8\pi}{5}$ has the conjugate angle $\frac{2\pi}{5}$. Since there exists no prime caterpillar with an angle of $\frac{2\pi}{5}$, an angle measuring $\frac{8\pi}{5}$ on one side and $\frac{2\pi}{5}$ on the other side can correspond only to a prime caterpillar forming an angle of $\frac{8\pi}{5}$. The only possible prime caterpillar for such an angle is $PC_4$. By Proposition~\ref{prop: alternance}, the rest of the star-graph path of $\mathcal{C}$ completely determines the saturated caterpillar $\mathcal{C}$. Hence, we can use a star-graph path containing at least one angle that measures $\frac{8\pi}{5}$ on one side to represent a fully leafed saturated caterpillar.

Just as we did with prime caterpillars, we find paths in the star-graphs that are primitive blocks to construct bigger fully leafed caterpillars. We call the caterpillars corresponding to these paths \textbf{sea caterpillars}, and we present them in Figure \ref{fig:chenilles marines}.

\section{Bi-infinite fully leafed caterpillars}

Since we now know how to graft caterpillars to construct large fully leafed caterpillars, the natural question that comes to mind is to ask if there are such caterpillars that can be extended infinitely. An infinite induced caterpillar $\mathcal{C}$ is said to be fully leafed if there exists a sequence $(\mathcal{C}_n)_{n\in \mathbb{N}}$ of finite fully leafed induced caterpillars such that for every $n\in \mathbb{N}$, $\mathcal{C}_n\subseteq \mathcal{C}_{n+1}$ and $\mathcal{C}=\lim_{n\to \infty} \mathcal{C}_n$. Since a caterpillar can be extended in two directions, we are interested in finding bi-infinite fully leafed caterpillars. A first one was found in \cite{porrier2023leaf}, and we formalize its construction in \cite{cloutier2026fullyleafedinducedsubtrees}. This bi-infinite caterpillar does not contain any cape 4. Conjecture 2 of \cite{porrier2023leaf} suggests that this bi-infinite caterpillar is the unique bi-infinite fully leafed caterpillar of P2-graphs. It is not a trivial problem, because any finite fully leafed induced caterpillar cannot necessarily be extended into a bi-infinite fully leafed caterpillar. To approach this conjecture, we focus on results that allow us to identify caterpillars that cannot be extended to bi-infinite fully leafed caterpillars.

\begin{figure}[H]
    \centering
    \begin{subfigure}{0.29\textwidth}
        \centering
        \includegraphics[width=\textwidth, trim=0cm 0.5cm 0cm 0.3cm, clip]{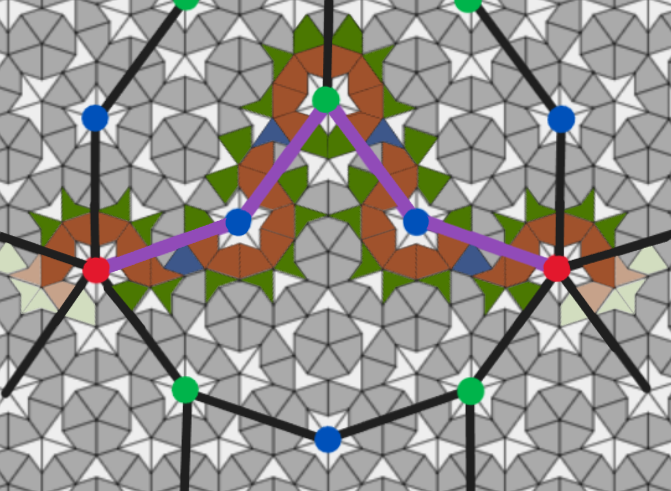}
        \caption{Sail}
    \end{subfigure}
    \begin{subfigure}{0.29\textwidth}
        \centering
        \includegraphics[width=\textwidth, trim=0cm 0cm 0cm 0cm, clip]{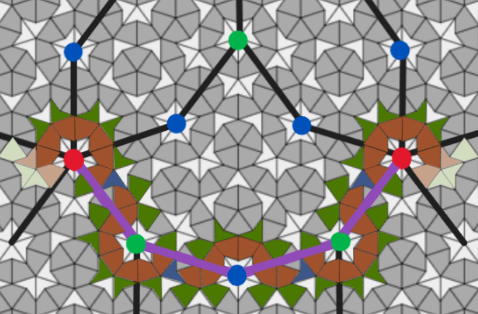}
        \caption{Hull}
    \end{subfigure} 
    \begin{subfigure}{0.25\textwidth}
        \centering
        \includegraphics[width=\textwidth, trim=0cm 0.4cm 0cm 0.1cm, clip]{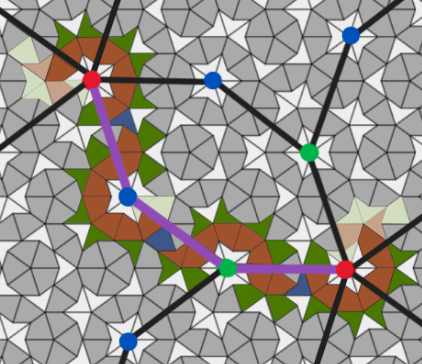}
        \caption{Dock 1}
    \end{subfigure}  
    \begin{subfigure}{0.26\textwidth}
        \centering
        \includegraphics[width=\textwidth, trim=0cm 0.9cm 0cm 0.05cm, clip]{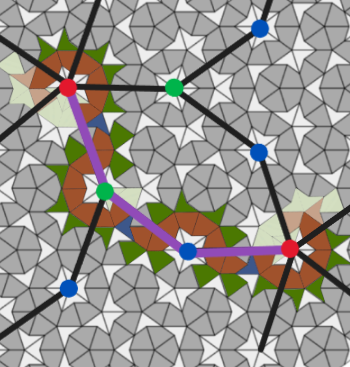}
        \caption{Dock 2}
    \end{subfigure}
    \begin{subfigure}{0.26\textwidth}
        \centering
        \includegraphics[width=\textwidth, trim=0cm 0.2cm 0cm 0cm, clip]{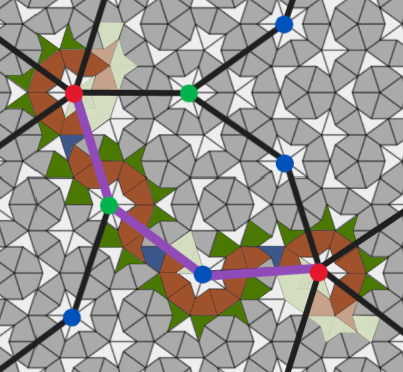}
        \caption{Dock 3}
    \end{subfigure}
    \begin{subfigure}{0.26\textwidth}
        \centering
        \includegraphics[width=\textwidth]{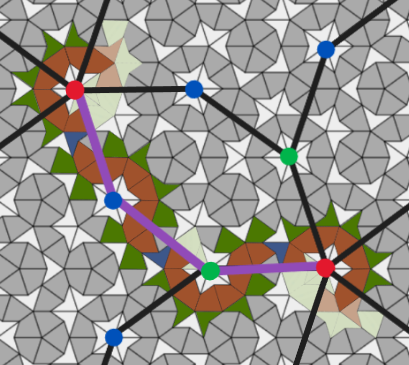}
        \caption{Dock 4}
    \end{subfigure} 
    \begin{subfigure}{0.31\textwidth}
        \centering
        \includegraphics[width=\textwidth, trim=0cm 0cm 0cm 0cm, clip]{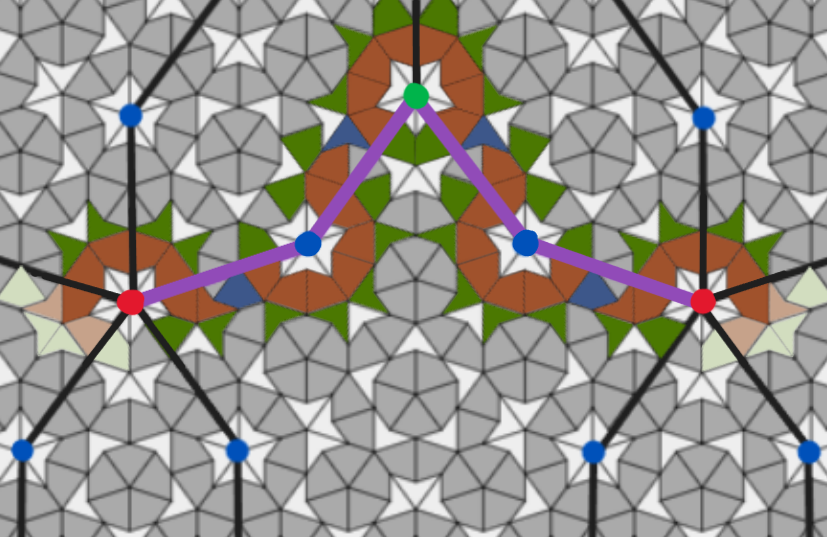}
        \caption{Cape 1}
    \end{subfigure} 
    \begin{subfigure}{0.42\textwidth}
        \centering
        \includegraphics[width=\textwidth, trim=0cm 0.7cm 0cm 0cm, clip]{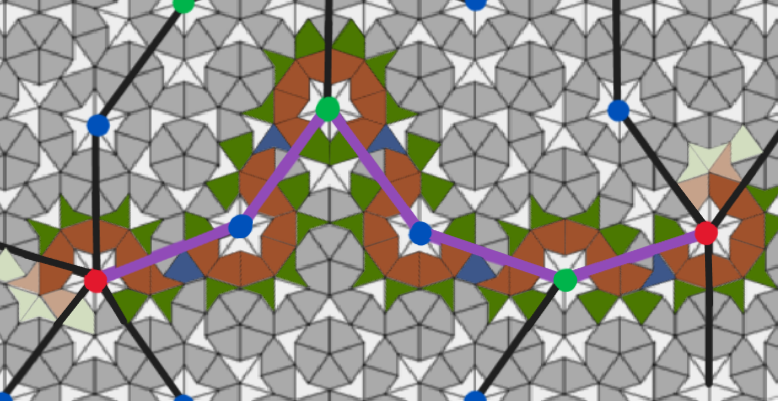}
        \caption{Cape 2}
    \end{subfigure} 
    \begin{subfigure}{0.42\textwidth}
        \centering
        \includegraphics[width=\textwidth, trim=0cm 0.4cm 0cm 0cm, clip]{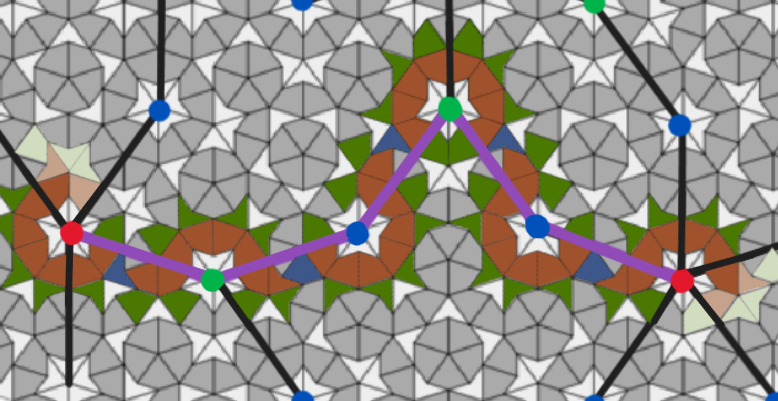}
        \caption{Cape 3}
    \end{subfigure} 
    \begin{subfigure}{0.465\textwidth}
        \centering
        \includegraphics[width=\textwidth, trim=0cm 0cm 0cm 0cm, clip]{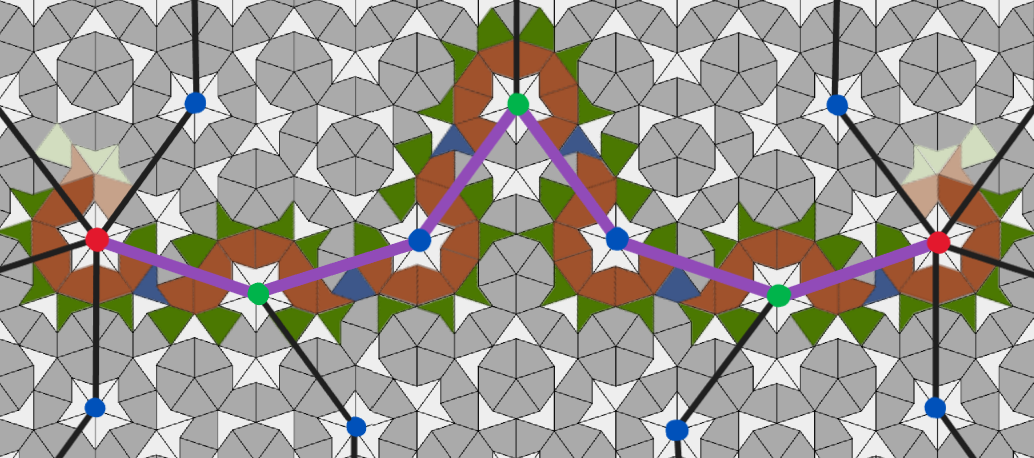}
        \caption{Cape 4}
    \end{subfigure}
    \caption{Sea caterpillars (up to rotation). We keep the color code of the previous figures.}
    \label{fig:chenilles marines}
   
\end{figure}

\begin{lemma}[\cite{cloutier2026fullyleafedinducedsubtrees}]
    There exists no bi-infinite fully leafed caterpillar that contains two consecutive $\frac{4\pi}{5}$ angles.
\end{lemma}

\begin{proposition}[\cite{cloutier2026fullyleafedinducedsubtrees}]
    The prime caterpillar $PC_1$ cannot belong to any bi-infinite fully leafed caterpillar.
\end{proposition}

\begin{proposition}[\cite{cloutier2026fullyleafedinducedsubtrees}]
    Neither a cape 2 nor a cape 3 sea caterpillar can belong to a bi-infinite fully leafed caterpillar.
\end{proposition}

Other similar results, in particular about graftings of sea caterpillars, can be found in \cite{cloutier2026fullyleafedinducedsubtrees}. These results allow us to construct a new bi-infinite fully leafed caterpillar.

\begin{theorem}[\cite{cloutier2026fullyleafedinducedsubtrees}]
    There exists a bi-infinite fully leafed caterpillar that contains a cape 4.
\end{theorem}
\begin{proof}
    The proof sketch is as follows. We consider an extension of a cape 4, that we call a \textbf{Super cape 4}, which is illustrated in Figure \ref{Super cape 4}. By inflating the region in which it is embedded, we obtain a new region of a P2 tiling where we can construct a larger fully leafed caterpillar that contains a Super cape 4, which we denote by $\psi(\text{Super cape 4})$ and show in Figure \ref{psi of Super cape 4}. We then proceed to find similar constructions for the new caterpillars in $\psi(\text{Super cape 4})$, and associate them with larger caterpillars in the third inflation of the tiling region where they are embedded. There are 8 such constructions to find, and they must be defined so that for all $n \in \mathbb{N}$, $\psi^n(\text{Super cape 4})$ remains connected and acyclic and is embedded into a valid patch.

\begin{figure}[h]
\centering
\begin{subfigure}{0.4\textwidth}
    \centering
   \includegraphics[width=\textwidth]{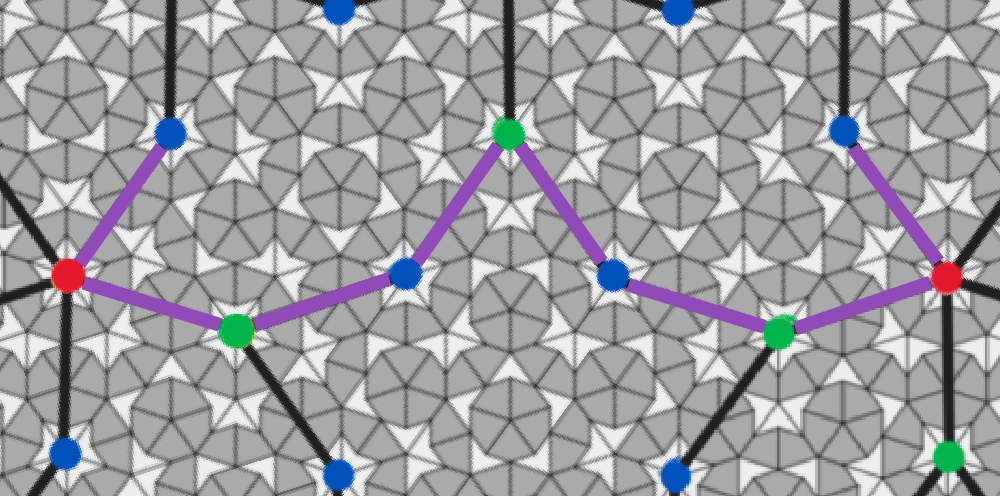}
\caption{Super cape 4}
\label{Super cape 4} 
\end{subfigure}
\begin{subfigure}{0.376\textwidth}
    \centering
\includegraphics[width=\textwidth]{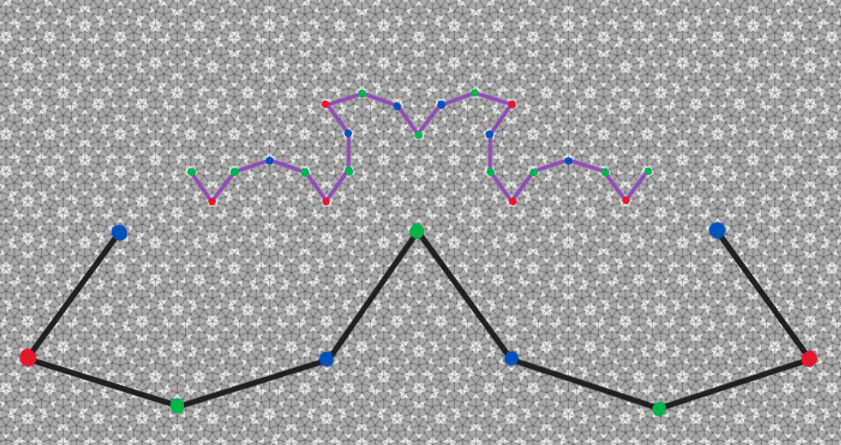}
\caption{$\psi(\text{Super cape 4})$}
\label{psi of Super cape 4}
\end{subfigure}
\caption{Super cape 4 and $\psi(\text{Super cape 4})$. In Figure (b), the black path is the star-graph chain of a Super cape 4 before applying three inflations and is used as a reference in the tiling and the purple path is the star-graph chain of $\psi(\text{Super cape 4})$.}
\end{figure}
    
\end{proof}

\section{Conclusion}

We now have a better understanding of the structure of fully leafed induced subtrees in P2 tilings. By Theorem \ref{thm2}, they are caterpillars up to an appendix of at most six tiles, and the saturated ones are caterpillars. We disproved Conjecture 2 from \cite{porrier2023leaf} that there is a unique bi-infinite fully leafed caterpillar in P2 tilings.

Since we now know that there exists more than one bi-infinite fully leafed caterpillar in Penrose P2 tilings, we aim to find them all. We have already eliminated several patterns. It is also possible to consider the star-graph chains of these bi-infinite structures as words over a 3-letter alphabet whose letters are either the colors of the vertices in these chains (red, green, and blue) or the angles of the prime caterpillars, and to study the properties of these words. We have begun this work in \cite{cloutier2026fullyleafedinducedsubtrees}. Of course, it would also be interesting to describe the fully leafed induced subtrees in other aperiodic tilings.\\

\noindent\textbf{Acknowledgements:} We thank Kevin Bertman for developing the website \url{https://www.mrbertman.com/penroseTilings.html}, which we used to construct patches of P2 tilings.

\nocite{*}
\bibliographystyle{eptcs}
\bibliography{bibliography}

\end{document}